\documentclass[10pt]{article}
\font\smallit=cmti10
\font\smalltt=cmtt10

\usepackage{amssymb,latexsym,amsmath,epsfig,amsthm} %% Add other packages as necessary

\makeatletter

\renewcommand\section{\@startsection {section}{1}{\z@}
{-30pt \@plus -1ex \@minus -.2ex}
{2.3ex \@plus.2ex}
{\normalfont\normalsize\bfseries\boldmath}}

\renewcommand\subsection{\@startsection{subsection}{2}{\z@}
{-3.25ex\@plus -1ex \@minus -.2ex}
{1.5ex \@plus .2ex}
{\normalfont\normalsize\bfseries\boldmath}}

\renewcommand{\@seccntformat}[1]{\csname the#1\endcsname. }

\makeatother

\newtheorem{Theorem}{Theorem}
\newtheorem{Lemma}{Lemma}

\newtheorem{Corollary}{Corollary}

\theoremstyle{Definition}
\newtheorem{Definition}{Definition}

\newtheorem{Remark}{Remark}
\newtheorem{Example}{Example}

\begin{document}

\begin{center}
\uppercase{\bf \boldmath Locally Integer Polynomial Functions}
\vskip 20pt
{\bf Alexander Borisov}\\
{\smallit Department of Mathematics and Statistics, Binghamton University, Binghamton, New York, USA}\\
{\tt alexanderborisov.math@gmail.com}\\ 
\end{center}
\vskip 20pt
%\centerline{\smallit Received: , Revised: , Accepted: , Published: } % We will fill in the dates
\vskip 30pt

\centerline{\bf Abstract}
\noindent
The goal of this note is to bring attention to an interesting family of rings: the rings of $\mathbb Z$-valued functions on $\mathbb Z$ and, more generally, infinite subsets of $\mathbb Z$ whose restrictions to all finite sets are given by polynomials with integer coefficients. Our interest in these functions was inspired by the work of Sayak Sengupta on iterations of integer polynomials, but they appear to be of independent interest. In particular, they enjoy some properties reminiscent of the properties of complex analytic functions, including forming a sheaf in the cofinite and density one topologies. 

\pagestyle{myheadings}
\markright{\smalltt INTEGERS: 24 (2024)\hfill}
\thispagestyle{empty}
\baselineskip=12.875pt
\vskip 30pt

\section{Introduction} 
Let $f\colon \mathbb Z \to \mathbb Z$ be a function. By a classical polynomial interpolation theorem, for any finite set $X\subset \mathbb Z$ there exists a polynomial $p(x) \in \mathbb Q[x]$ such that $p(x)=f(x)$ for all $x\in X$. Moreover, if we require $\deg p(x) < |X|,$ this $p(x)$ is unique; we will denote it by $f_X(x)$. It is well known that in general one cannot guarantee that $f_X\in \mathbb Z[x]$: $f(x)=\frac{x(x+1)}{2}$ provides a simple, perhaps the simplest, counterexample.  If we require that for all finite $X$ the polynomial $f_X$ does belong to $\mathbb Z[x]$, we get infinitely many congruence conditions on the values of $f$. To begin with,  for all integers $n$ and $m$  by considering $X=\{n,m\}$  we get that $n-m$ must divide $f(n)-f(m)$. More restrictive, but also more complicated, conditions appear when we consider $X$ with $|X|\geq 3$. For the lack of a better name, we will call such functions locally integer polynomial functions or LIP functions. Clearly, these functions form a ring, that we will denote by $L\mathbb Z[x]$ or $LIP(\mathbb Z)$. More generally, replacing $\mathbb Z$ by its arbitrary infinite subset $U$, we get the ring $LIP(U)$ of locally integer polynomial functions on $U$. 

Our initial motivation to study these objects comes from the paper of Sayak Sengupta \cite{Sengupta2} on iterations of integer polynomials. Some of the main results of that paper concern with partial classification of orbits of integer polynomials starting at a fixed integer and eventually containing $0$. In particular, he found the maximal length of such sequences, depending on the starting integer, together with some other finiteness and structural results. It is easy to see that his results can be extended to iterations of LIP functions, as his methods only involve integer polynomials and finite integer sequences. To preempt any confusion, we must point out that our use of the word ``locally'' is different from  Sengupta's use of it in \cite{Sengupta1}, where it means ``modulo primes''.

It turned out that LIP functions on $\mathbb Z$ and, especially, on infinite subsets of $\mathbb Z$ are rich and interesting objects. In particular, there are some mysterious similarities between these rings and the rings of complex-analytic functions on domains in $\mathbb C$. As such, this theory has some potential to fill in an interesting gap in our understanding of the integers. There have been many attempts to develop a theory for the integers that would resemble the usual topology on $\mathbb C,$ viewed as the set of maximal ideals on $\mathbb C[x]$, including, in particular, the theory of Berkovich spaces and various results around the technically non-existent ``field of one element'' (e.g. \cite {BPR, Berkovich, JaiungJun, Lorscheid}). We hope that LIP functions become an important addition to these efforts, especially if combined with some natural topologies on $\mathbb N$, like Golomb and Kirch topologies (\cite{Clarketal, Golomb, Kirch, LMS}).  Golomb and Kirch topologies have been around for a long time, but are relatively unknown to number theorists. They are Hausdorff and connected, and are defined using certain arithmetic progressions. As we will see in Sections 4 and 5, arithmetic progressions naturally appear in the theory of LIP functions.

In this direction we prove that LIP functions form a sheaf, and not just a presheaf, for cofinite and density one topologies on $\mathbb Z$ and $\mathbb N$. For Golomb and Kirch topologies, the situation is more complicated, because they have non-connected open sets, and for a non-connected open set the sheafification of the LIP presheaf must be a direct product of the sheafifications on the connected components. We consider the Kirch topology and present some partial results and open questions regarding LIP functions on Kirch-open sets.

One would also hope that there is a connection with Diophantine equations. Obviously, if a system of Diophantine equations has a non-constant solution in $LIP(U)$ for some infinite $U\subseteq \mathbb Z$, then it has infinitely many solutions in $\mathbb Z$. The converse is unlikely to hold in general, and we do not  know if it is true even for the classical Pell's equation $x^2-2y^2=1.$ 

The paper is organized as follows. Section 2 is devoted to the general results for the rings of LIP functions on infinite subsets of $\mathbb Z$. In particular, we prove that these rings are never Noetherian. In Section 3 we study $LIP(\mathbb Z)$ with the special emphasis on growth estimates for non-polynomial LIP functions. In Section 4 we consider the problem of extending a LIP function from one infinite subset to another (LIP continuation). This leads us to a question of when an integer-valued function is not LIP, and we study the ``minimal obstructions'' to LIPness, that we call circuits. In the final Section 5 we study the sheaf-theoretic properties of the rings of LIP functions. The main result of this section is Theorem 9 that shows that a function is LIP if its restriction to each subset in some covering family is LIP, provided that the intersection of any two subsets in the family is ``large enough''. This implies that LIP functions form not just a presheaf, but a sheaf in some natural non-Hausdorff topologies. We also define and study the sheafification of the presheaf of LIP functions, the ``locally LIP'' functions, with respect to the Kirch topology on $\mathbb N$ and discuss some open questions regarding this notion.

%%%%%%%%%%%%%%%%%%%%%%%%%%%%%%%%%%%%
\section{General results on LIP functions on infinite subsets of $\mathbb Z$}

We start with the basic definitions.

\begin{Definition} Suppose $U$ is an infinite subset of $\mathbb Z$. Then $f\colon U \to \mathbb Z$ is a locally integer polynomial (LIP, for short) function if the restriction of $f$ to any finite subset $X$ of $U$ can be given by a polynomial with integer coefficients. 
\end{Definition}

\begin{Definition} For an infinite subset $U$ of $\mathbb Z$, we denote by $LIP(U)$ the subset of all functions $f\colon U \to \mathbb Z$ that consists of all  LIP functions on $U$.
\end{Definition}

The following lemma is immediate.

\begin{Lemma}
Suppose $U\subseteq \mathbb Z$ is infinite. Then $LIP(U)$ is a subring of the ring of all integer-valued functions on $U$. It is a commutative ring with identity, containing $\mathbb Z[x]$.
\end{Lemma}

\begin{proof} As a subset of the ring of all integer-valued functions on $U$, the set $LIP(U)$ is closed under subtraction and multiplication. Since polynomials are uniquely determined by their values on the infinite set $U$, the natural evaluation map from $\mathbb Z[x]$ to the ring of all integer-valued functions on $U$ is an inclusion. And, clearly, the image of this map is in $LIP(U)$.
\end{proof}

It is logically possible that even if the interpolation polynomial for the function $f$ on a finite set $X$ is not an integer polynomial, there exists some other polynomial $p(x) \in \mathbb Z[x]$, of higher degree, that takes the same values on all $x\in X$. However the following easy lemma shows that this cannot happen (cf. also \cite{Sengupta2}, Lemma 2.2). 

\begin{Lemma} Suppose $X=\{x_0,\ldots ,x_d\}\subset \mathbb Z$. Suppose $f_X(x)\in \mathbb Q[x]$ is a polynomial of degree at most $d,$ and $p(x)\in \mathbb Z[x]$ is such that   $p(x)=f_X(x)$ for all $x\in X$. Then $f_X(x)\in \mathbb Z[x]$.
\end{Lemma}

\begin{proof} Clearly, $p(x)=f_X(x)+g(x)\cdot \prod \limits_{i=0}^d (x-x_i)$ for some $g(x)\in \mathbb Q[x]$. If $g(x) \notin \mathbb Z[x],$ denote by $k$ the largest degree of $x$ in $g(x)$ with non-integer coefficient $g_k$. Since the coefficient of $x^{k+d+1}$ in $p(x)$ also equals $g_k$, we see that $p(x) \notin \mathbb Z[x]$, a contradiction.  Thus $g(x)\in \mathbb Z[x]$, so $f_X(x)\in \mathbb Z[x]$.
\end{proof}

The above lemma suggests the following definition.

\begin{Definition} For a function $f\in LIP(U)$ and finite $X\subset U$, we call the interpolation polynomial for $f$ on $X$ the restriction of $f$ to $X$, to be denoted by $f_X$, or $f_X(x)$ is we want to specify that we consider it as an element of $\mathbb Z[x]$. We will also use the same convention for arbitrary $f\colon U \to \mathbb Z$, except that $f_X$ may be a rational polynomial.
\end{Definition}

The following theorem shows that $LIP(U)$ is a much bigger ring than $\mathbb Z[x]$. In fact, it has the same cardinality as $\mathbb R$.

\begin{Theorem} Suppose $\sigma: \mathbb N \to U$ is any bijection. Then
 for any $a_0,a_1,a_2,\ldots \in \mathbb Z$ the following function is in $LIP(U)$:
 $$f(x)=\sum \limits_{k=0}^{\infty}a_k\prod \limits_{i=1}^k (x-\sigma(i)) = a_0+a_1(x-\sigma(1))+\ldots +a_k\prod \limits_{i=1}^k (x-\sigma(i))+\ldots $$
Moreover, every function in $LIP(U)$ is of this form for some unique sequence of integers $\{a_i\}_{i=0}^{\infty}$.
\end{Theorem}

\begin{proof} Even though the sum is formally infinite, all terms of its restriction to $x\in U$ vanish for $k\geq \sigma^{-1}(x)$. Thus this formula does define an integer-valued function on $U$ and, moreover, it is in $LIP(U)$. To prove the existence and uniqueness, denote by $f_n(x)$ the partial sum of the above series: $f_n(x)=\sum \limits_{k=0}^{n-1}a_k\prod \limits_{i=1}^k (x-\sigma(i)).$ Then for every $m=\sigma(n)\in U$ we need to have
$$f(m)=f_{n}(m)= f_{n-1}(m)+a_{n-1}\prod \limits_{i=1}^{n-1} (m-\sigma(i)).$$
This is equivalent to $$a_{n-1}=(f(m)-f_{n-1}(m))/ \prod \limits_{i=1}^{n-1} (m-\sigma(i)).$$
Thus for every $f\colon U \to \mathbb Z$ we can recursively define the unique sequence of rational $a_i$ such that $f$ is given by the above infinite series. And if $f$ is a LIP function then all polynomials $f_n$ are in $\mathbb Z[x],$ so all $a_i$ are integers. 
\end{proof}

\begin{Remark}
Different $\sigma$ produce different sequences of coefficients for the same function $f$. And while this description of LIP functions is good for their addition, it is not well suited for the product.
\end{Remark}

\begin{Remark}
Instead of the sums above, we can consider formal sums
$f(x)=\sum \limits_{k=0}^{\infty}a_k(x) \prod \limits_{i=1}^k p_i(x)^k$, where the sequence $\{p_n(x)\}$ runs through the (countable) set of all irreducible integer polynomials with content $1$ and positive leading coefficient, and $a_k(x)$ are integer polynomials. Such functions are defined not just on $\mathbb Z$, but on $\mathbb Q^{alg}$. Moreover, all their formal derivatives are also defined on $\mathbb Q^{alg}$ and have the property that all restrictions to finite subsets are given by integer polynomials. Perhaps, such generalizations are also worth studying.
\end{Remark}

The following observation is vaguely reminiscent of a well-known important property of complex analytic functions.

\begin{Lemma} Suppose $f \in LIP(U)$ and $f(x)=0$ for infinitely many $x\in U$. Then $f=0$.
\end{Lemma}

\begin{proof} Suppose $X=\{x\in \mathbb Z\ |\ f(x)=0\}$. Take arbitrary $m\in U$. Since $X$ is infinite, we can find  $x\in X$ such that $|x-m| > |f(m)|$. Since $|x-m|$ must divide $|f(x)-f(m)|=|f(m)|$, we must have  $f(m)=0$. Thus $f=0$.
\end{proof}

\begin{Corollary} Suppose $U$ and $V$ are infinite subsets of $\mathbb Z$ and $U\subseteq V$. Then the restriction map $LIP(V)\to LIP(U)$ is an injection. In particular, $LIP(\mathbb Z)$ is naturally contained in all $LIP(U)$. 
\end{Corollary}

\begin{Lemma} For every infinite $U\subseteq \mathbb Z$ the ring $LIP(U)$ is a domain, and its only units are $1$ and $-1$.
\end{Lemma}

\begin{proof} If $f(x)g(x)=0$ for all $x\in U$, then $f$ or $g$ must vanish on an infinite subset of $U$, so by Lemma 3 it must be zero.
If $f(x)g(x)=1$  for all $x\in U$, then $f(x)=\pm 1$ for all $x$. Therefore $(f-1)(f+1)=0,$ so $f=1$ or $f=-1$.
\end{proof}

We can take this one step further, with the following easy theorem.

\begin{Theorem} Suppose $U\subseteq \mathbb Z$ is infinite. If $f\in LIP(U)$ and $|f(x)|$ is prime for infinitely many $x$, then $f$ is irreducible in $ LIP(U)$.
\end{Theorem}

\begin{proof} If $f=gh,$ then one of the following must be true for infinitely many $x\in U$: $g(x)=1,$ $g(x)=-1$, $h(x)=1,$ $h(x)=-1$. In the first two cases $g$ is a unit, and in the last two cases $h$ is a unit.
\end{proof}

So far we have only used the property of LIP functions that the difference of inputs must divide the difference of outputs. The proof of the next theorem uses the LIP property in a more substantial way.

\begin{Theorem} For $C \in \mathbb{R}$ and $d \in \mathbb{N}$, suppose $f \in LIP(U)$ is such that $|f(x)|\leq C|x|^d$ for infinitely many $x\in U.$ Then $f\in \mathbb Z[x]$.
\end{Theorem}

\begin{proof} Suppose $X=\{ x\in U \ |\ |f(x)|\leq C|x|^d \}.$  Fix pairwise distinct $x_0,x_1,\ldots,x_d$ in $X$. Suppose $g(x)=f_{\{x_0,x_1,\ldots,x_d\}}(x)$. Suppose $x_{d+1}\in X$ and $|x_{d+1}|$ is large enough. Then by adding $x_{d+1}$ to the list of inputs, and considering the new interpolation polynomial, we get that 
$$f(x_{d+1})=a\cdot  \prod \limits_{i=0}^{d} (x_{d+1}-x_i)+g(x_{d+1})$$
for some  $a\in \mathbb Z$ that depends on $x_{d+1}$. Now note that as $|x_{d+1}|$ goes to $+\infty$, $|f(x_{d+1})|$ and $|g(x_{d+1})|$ grow at most as a constant multiple of $|x|^d$, while $\prod_{i=0}^{d} (x_{d+1}-x_i)$ grows like $|x|^{d+1}$. Therefore for all  $x_{d+1}$ in $X$ with large enough $|x_{d+1}|$  we must have $a=0$, thus $f(x_{d+1})=g(x_{d+1})$. Thus $f(x)-g(x)=0$ for infinitely many $x$, so by Lemma 3 $f(x)=g(x)$ for all $x$.
\end{proof}

\begin{Corollary} Suppose $f(x)$ is a prime element of $\mathbb Z[x],$ i.e. a prime in $\mathbb Z$ or an irreducible integer polynomial with content 1. Then it is a prime element of $LIP(U)$.
\end{Corollary}

\begin{proof} If $f(x)=g(x)h(x)$ then $|g(x)|$ and $|h(x)|$  grow at most polynomially as $|x|$ goes to infinity. So  $g(x)$ and $h(x)$ are in $\mathbb Z [x].$
\end{proof}

\begin{Corollary} $\mathbb Z[x]$ is algebraically closed in $LIP(U)$. 
\end{Corollary}

\begin{proof} Suppose $f\in LIP(U)$ satisfies some algebraic equation with coefficients in $\mathbb Z[x].$
Then by looking at this equation over $\mathbb R$, the absolute value of $f(x)$  must grow at most polynomially in $|x|$.
\end{proof}

There is no reason to expect the rings $LIP(U)$ to be Noetherian. Indeed, we have the following result.

\begin{Theorem} Suppose $U$ is an infinite subset of $\mathbb Z$. Then $LIP(U)$ contains an infinite strictly increasing chain of principal ideals.
\end{Theorem}

\begin{proof} While all LIP functions can be expressed as formal infinite sums, some can actually be expressed as formal infinite products. Specifically, for any bijection $\sigma \colon \mathbb Z \to U$ and for all $m\in \mathbb N$ consider the functions
$$f_m(x)=\prod \limits_{k=m}^{\infty} \left(1+\prod \limits_{i=1}^k (x-\sigma(i)) \right).$$

Clearly, $f_m\in LIP(U)$ for all $m$, and we have an infinite strictly increasing chain of principal ideals
$$(f_1)\subset (f_2) \subset \cdots \subset (f_m)\subset \cdots.$$
\end{proof}

\begin{Remark} It seems plausible that the functions $f_m$ cannot be written as products of finitely many primes (irreducibles), which would mean that $LIP(U)$ are never atomic (\cite {AAZ}). However, we do not have a proof at this time.
\end{Remark}

%%%%%%%%%%%%%%%%%%%%%%%%%%%%%%%%%%%
\section{LIP functions on $\mathbb Z$}

In this section we study $LIP(\mathbb Z)= L\mathbb Z[x].$ As noted in the previous section, $L\mathbb Z[x]$ is a commutative ring with identity. It contains $\mathbb Z[x]$ and is contained in $LIP(U)$ for any infinite $U\subseteq \mathbb Z$. We are primarily concerned with the behavior of LIP functions w.r.t. the ``infinite prime'', i.e. the regular absolute value. A very similar theory can also be developed for $LIP(\mathbb N)$. 

The following two properties are specific for $L\mathbb Z[x]$

\begin{Theorem}
1)  $L\mathbb Z[x]$ is closed under the operation of composition of functions.

2)  $L\mathbb Z[x]$ is closed under the (left) discrete derivative operator: if $f(x) \in L\mathbb Z[x]$, then $f(x)-f(x-1) \in L\mathbb Z[x]$.

\end{Theorem}

\begin{proof} 1) If $f(x)$ and $g(x)$ are in  $L\mathbb Z[x]$ then the restriction of their composition $f(g(x))$ to any finite set $X$ is given by the polynomial $f_Y(g_X(x))$, where $Y=g(X)$.

2) Suppose $f(x)$ is in $L\mathbb Z[x]$. From the first part of the theorem, $f(x-1)$ is also in $ L\mathbb Z[x].$ So the same is true for the difference $f(x)-f(x-1)$.
\end{proof}

%\begin{Remark}
%One can view the ring $L\mathbb Z[x]$ as the closure of $\mathbb Z[x]$ in the space of all integer-valued functions on $\mathbb Z$ in the topology of the point-wise convergence, with the trivial norm on $\mathbb Z$. If one considers instead the $p$-adic norm, for some prime $p$, the resulting ring is some intermediate ring between the ring of power series $\mathbb Z_p[[x]]$ that converge at $1$ and the whole ring $\mathbb Z_p[[x]]$. It may be interesting in its own right. 
%\end{Remark}

%\begin{Remark} Every LIP function $f:\mathbb Z \to \mathbb Z$ can be extended by continuity to a  function $\overline{f}:\overline{\mathbb Z}\to \overline{\mathbb Z}$, where  $\overline{\mathbb Z}$ is the profinite completion of $\mathbb Z$, also known as the ring of finite adeles. Moreover,  under the isomorphism $\overline{\mathbb Z}\cong \prod _p \mathbb Z_p$ the function $\overline{f}$ is just the direct product of some functions $f_p:\mathbb Z_p \to \mathbb Z_p$ and these functions $f_p$ are Lipschitz in the p-adic norm. NEED TO EXPAND.
%\end{Remark}

By Theorem 3, if for $f\in L\mathbb Z[x]$ the absolute value of $f(x)$ grows at most polynomially on some infinite subset of $\mathbb Z$, then $f$ is just a polynomial. 
In some sense, this result is tight, as the next theorem shows.

\begin{Theorem} If $\tau: \mathbb N \to \mathbb R$ is any sequence with $\lim \limits_{n\to +\infty} \tau (n)= +\infty,$ then there exists $f\in L\mathbb Z[x] \setminus \mathbb Z[x]$  such that $|f(x)|<|x|^{\tau (|x|)}$ for all $x$ in some infinite  $X \subseteq \mathbb Z$.
\end{Theorem}

\begin{proof} For every strictly increasing sequence of natural numbers $\{c_1,c_2,\ldots \}$ consider the LIP function
$$f(x)=\sum \limits_{i=1}^{\infty} \prod \limits_{|j|\leq c_i} (x-j)$$
and $X=\{c_{k+1}-1\ |\ k\in \mathbb N\}.$

Then for every $x=c_{k+1}-1 \in X$ 
$$|f(x)|\!=\!|\sum \limits_{i=1}^k  \prod \limits_{|j|\leq c_i} \! (x-j) |\!\leq \! k\! \prod \limits_{|j|\leq c_k}\! |x-j|\!\leq \! k(c_{k+1}-1+c_k)^{2c_k+1}\!\leq\!(2x)^{2c_k+2}.$$

Note that this is less that $x^{4c_k+4}$. So for a given $\tau$ we will choose $c_1=1$ and then for each $k$ we will be able to choose $c_{k+1}>c_k$ so that $\tau(c_{k+1}-1)>4c_k+4$. Clearly, $|f(x)|<|x|^{\tau (|x|)}$ for all $x\in X,$ the set $X$ is infinite, and $f\notin \mathbb Z[x].$
\end{proof}

\begin{Remark} The above theorem also holds for $LIP(U)$ for arbitrary infinite $U$, simply because $L\mathbb Z[x]\subseteq LIP(U)$.
\end{Remark}

On the other hand, Theorem 3 does admit a non-trivial strengthening, if instead of looking at some infinite $X\subseteq \mathbb Z$ we look at the entire $\mathbb Z.$ For this, it is convenient to first prove a lemma.

\begin{Lemma} Suppose $f\in \mathbb Z[x]$ and $\deg f=d.$ Then $|f(x)|\geq \frac{d!}{2^d}$ for some $x\in \{0,1,\dots,d\}$. 
\end{Lemma}

\begin{proof} Consider the polynomial $g(x) $ of degree at most $(d-1)$ such that $g(x)=f(x)$ for all $x\in \{0,1,\dots,d-1\}$. By Lemma 2, $g\in \mathbb Z[x]$. Clearly, $f(x)=g(x)+ax(x-1)\dots(x-d+1)$ for some integer $a$. Since $\deg f =d,$ we have $a\neq 0.$ Consider the (left) discrete derivative operator:  $(\delta f)(x)=f(x)-f(x-1).$ For every $r<s\in \mathbb N$ this can be defined on functions on $\{r,\dots,s\}$ with outputs in functions on $\{r+1,\dots,s\}.$  So its $d$-th composition power, $\delta ^d$, applied to $f(x)$ on $\{0,\dots,d\}$, gives a function on $\{d\}$. By a standard calculation, $\delta ^d (g)=0$ and $(\delta ^d f)(d)=d!\cdot a$. If for all $x$ in $\{0,\dots,d\}$ we have $|f(x)|<\frac{d!}{2^d},$ then for all $x$ in $\{1,\dots,d\}$ we must have $|(\delta f)(x)|<\frac{d!}{2^{d-1}},$ and so on. At the end, $|(\delta ^d f)(d)| < d!$, a contradiction with $a\neq 0$.
\end{proof}

%\begin{Theorem} For every non-polynomial LIP function $f(x)$, 
%$$\limsup \limits_{k\to +\infty} \  (|f(x)|\cdot \frac{2^{2|x|-1}}{(2|x|-1)!}) \geq 1.$$
%\end{Theorem}

\begin{Theorem} Suppose $f(x)\in L\mathbb Z[x]$ and for all but finitely many $x\in \mathbb Z$ 
$$|f(x)| < \frac{(2|x|-1)!}{2^{2|x|-1}}.$$
Then $f\in \mathbb Z[x]$.
\end{Theorem}

\begin{proof} Consider the ``standard'' bijection  $\sigma : \mathbb N \to \mathbb Z$:
$$\sigma (k)=\left\{ \begin{array}{l}
0, \ \textrm{if}\  k=1\\
\frac{k}{2}, \ \textrm{if}\ k \ \textrm{is even}\\
-\frac{k-1}{2}, \ \textrm{if}\ k \geq 3\ \textrm{is odd}
\end{array}\right.$$

Denote the corresponding subsets $\sigma (\{1,2,\ldots,i\})$ by $X_i.$ Define $d_i$ to be the degree of the interpolation polynomial $f_i(x)$ for $f(x)$ on $X_i$. Clearly, $d_{i-1}\leq d_i$

Suppose $f\notin \mathbb Z[x].$ Then $\lim (d_i)=+\infty $ and $d_{i-1} < d_i$  for infinitely many $i$. We will prove that $d_i=i-1$ for all such $i$. Indeed, $f_i(x)=f_{i-1}(x) +a \prod \limits_{j\in X_{i-1}} (x-j)$. Since  $d_{i-1} < d_i$, we  have $a\neq 0$. Because $d_{i-1}\leq |X_{i-1}|-1,$ we have $d_i=|X_{i-1}|=i-1$. For such $i$, we can apply Lemma 5 to a suitable shift of $f_i,$ to get that $\max \limits_{j\in X_i} |f(j)| \geq \frac{(i-1)!}{2^{i-1}}$.  So for each $i$ from some infinite subset of $\mathbb N$ we get a $j\in X_i$ such that $|f(j)| \geq \frac{(i-1)!}{2^{i-1}}.$  Note  that $i\geq 2|j|$ for every $j\in X_i,$ so for all $|j| \geq 1$ we have $|f(j)|\geq \frac{(2|j|-1)!}{2^{2|j|-1}}$.  Note also that each $j$ can only work for finitely many $i$-s. Thus, if $f\notin \mathbb Z[x],$ we get an infinite set of $j$-s  for which $|f(j)|\geq \frac{(2|j|-1)!}{2^{2|j|-1}}$.
\end{proof}

The next example shows that the above bound cannot be improved much.

\begin{Example} Consider $f(x)=\sum \limits_{i=0}^{\infty}(-1)^i \prod \limits_{|j|\leq i} (x-j)$. Then $|f(x)|\leq (2|x|-1)!$ for all $x\in \mathbb Z.$
\end{Example}

\begin{proof} For $x=0$ the inequality is $0\leq 0.$ Since $f$ is clearly an odd function, it is enough to consider $x>0$. Then 
$$f(x)=(-1)^x\cdot [ (2x-1)\cdot \ldots \cdot 1 - (2x-2)\cdot \ldots \cdot 2+(2x-3)\cdot \ldots \cdot 3-\ldots+(-1)^x x].$$

Since the absolute value of the first term is $(2x-1)!$ and the absolute values of the terms in the alternating sum are decreasing, we get the desired inequality.
\end{proof}

%%%%%%%%%%%%%%%%%%%%%%%%%%%%%%%%%%
\section{LIP continuation and obstructions to LIPness}

We begin with a natural definition.

\begin{Definition}
Suppose $U$ and $V$ are infinite subsets of $\mathbb Z$. If $f\in LIP(U)$ is a restriction of $g\in LIP(V)$ to $U,$ then we say that $g$ is a LIP continuation of $f$ from $U$ to $V$.
\end{Definition}

\begin{Remark} LIP continuation does not always exist. But if it exists, it is unique by Corollary 1.
\end{Remark}

It is natural to ask if there exist LIP functions on $U$ that can be LIP continued to $U\sqcup \{a\}$ and $U\sqcup \{b\}$, but not to $U\sqcup \{a,b\}$. Because of the uniqueness of the LIP continuation, it is a question of finding a function on $U\sqcup \{a,b\}$ which is not LIP, but whose restrictions to  $U\sqcup \{a\}$ and $U\sqcup \{b\}$ are LIP. 

We are going to construct such infinite $U$ by ``growing'' it from the following finite example.

\begin{Example} Suppose $U=\{0\}$, $a=-1,$ $b=1$. Consider $f\colon \{-1,0,1\}\to \mathbb Z$ given by the values $\{0,0,1\}$ respectively. It is given by an integer-valued non-integer polynomial $f(x)=\frac{x(x+1)}{2}$, but its restrictions to $\{-1,0\}$ and $\{0,1\}$ can be given by integer polynomials $0$ and $x$ respectively.
\end{Example}

\begin{Theorem} Suppose $x_1,\dots,x_n,a,b$ are distinct integers, and $f(x)\in \mathbb Z[x]$ is an integer-valued non-integer polynomial such that its restrictions to  $\{x_1,\dots,x_n,a\}$ and  $\{x_1,\dots,x_n,b\}$ can be  given by integer polynomials. Then

1) $|a-b|\geq 2$

2) Denote by $V$ the set of all integers $x$, such that $gcd(x-a,a-b)=1.$ There exists a LIP function on the union $U$ of $\{x_1,\dots,x_n\}$ and $V$ that can be LIP continued to  $U \sqcup \{a\}$ and $U\sqcup \{b\}$, to coincide with $f(x)$ on $\{x_1,\dots,x_n,a,b\}$.
\end{Theorem}

\begin{proof} 
1) Denote by $X$ the finite set $\{x_1,\dots,x_n\}$. By the given conditions on $f$,

$$\left\{\begin{array}{l}
f_{X\sqcup \{a\}}(x)=f_X(x)+\alpha \cdot \prod \limits_{i=1}^{n}(x-x_i)\\
f_{X\sqcup \{b\}}(x)=f_X(x)+\beta \cdot \prod \limits_{i=1}^{n}(x-x_i)
\end{array} \right.$$ 
for some $\alpha \neq \beta$ in $\mathbb Z$.

Similarly, $f_{X\sqcup \{a,b\}}(x)=f_X(x)+(ux+v)\cdot \prod \limits_{i=1}^{n}(x-x_i)$, and we are given that $f_{X\sqcup \{a,b\}}(x) \notin \mathbb Z[x],$ so $u$ or $v$ is not an integer. Plugging in $a$ and $b$, we get
$\alpha=ua+v$ and $\beta=ub+v$. Thus $\alpha-\beta=u(a-b)$. If $a-b=\pm 1,$ then $u\in \mathbb Z.$ So, since $\alpha \in \mathbb Z,$ we get that $v\in \mathbb Z,$ a contradiction.

2) To prove the second statement, we will show that one can add to $X$ any element $x_{n+1}$ of $U\setminus X$, by defining $f(x_{n+1})$ appropriately. And we will repeat the argument one-by-one for all elements of $U$ that are not already in $X$. 

To define $f(x_{n+1})$ as $y_{n+1},$ we just need the following, for some integers $c$ and $d$:

$$\left\{\begin{array}{l}
y_{n+1}=f_{X\sqcup \{a\}}(x_{n+1})+c \cdot  (x_{n+1}-a) \prod \limits_{i=1}^{n}(x_{n+1}-x_i)\\
y_{n+1}=f_{X\sqcup \{b\}}(x_{n+1})+d \cdot (x_{n+1}-b) \prod \limits_{i=1}^{n}(x_{n+1}-x_i)
\end{array} \right.$$ 

Since we are free to choose any $y_{n+1},$ we just need to find $c$ and $d$ so that $$f_{X\sqcup \{a\}}(x_{n+1})+c(x_{n+1}-a) \prod \limits_{i=1}^{n}(x_{n+1}-x_i)=f_{X\sqcup \{b\}}(x_{n+1})+d (x_{n+1}-b) \prod \limits_{i=1}^{n}(x_{n+1}-x_i).$$

Subtracting the formulas in the proof of the first statement, and plugging in $x_{n+1}$ for $x$, we get $f_{X\sqcup \{a\}}(x_{n+1})-f_{X\sqcup \{b\}}(x_{n+1})=(\alpha-\beta)\prod \limits_{i=1}^{n}(x_{n+1}-x_i).$ So, dividing by $\prod \limits_{i=1}^{n}(x_{n+1}-x_i),$ we just need to solve the equation $\beta-\alpha=c(x_{n+1}-a)-d(x_{n+1}-b),$ for given $\alpha$ and $\beta$ and unknown $c$ and $d.$  Because $x_{n+1}\in V,$ $gcd(x_{n+1}-a, x_{n+1}-b)=1,$ so such $c$ and $d$ always exist.
\end{proof}

\begin{Remark} The above calculations also show that one cannot add to $X$ any element of the doubly-infinite arithmetic progression $a+(b-a)\mathbb Z$. Indeed, for such $x$ we have $gcd(x-a,x-b)=(a-b)$, which would imply that $\alpha-\beta$ is a multiple of $a-b$, which is impossible by the argument at the end of the proof of the first statement of the theorem.
\end{Remark}

\begin{Example} Applying the above theorem to Example 2, we get an example of a non-LIP function on $V=2\mathbb Z \cup \{-1,1\}$ which is LIP on $V\setminus \{-1\}$ and on  $V\setminus \{1\}$.
\end{Example}

\begin{Remark}The above example may be interpreted as evidence that the set $V=2\mathbb Z \cup \{-1,1\}$ is in some sense ``not simply-connected''. An analogous fact in the theory of complex-analytic functions is that $\sqrt{z}$ can be analytically continued from the domain $Re(z)>0$ to the complements in $\mathbb C$ of $i\mathbb R_{>0}$ and of $-i\mathbb R_{>0}$, but not to their union, $\mathbb C\setminus \{0\}$.
\end{Remark}

We also want to consider general integer-valued functions on subsets of $\mathbb Z$ that may or may not be LIP. Clearly, if a function is LIP on some set, it is LIP on all of its subsets. The following definition describes the ``minimal obstructions'' to the LIP condition. The terminology is inspired by the  theory of matroids. 

\begin{Definition} Suppose $f\colon U\to \mathbb Z$ is a function {\bf not} in $LIP(U)$. A finite subset $X$ of $U$ is called a {\bf circuit} for $f$ if $f_X\notin \mathbb Z[x] $ but $f_{X\setminus \{a\}}\in \mathbb Z[x]$ for every $a\in X$. 
\end{Definition}

Before analyzing the structure of circuits, it is convenient to prove a general lemma about interpolating polynomials.

\begin{Lemma} (Exchange Formula) Suppose $X\sqcup \{a,b\}$ is a finite subset of $\mathbb Z$, and $f$ is a function on it. Then
$$f_{X\sqcup \{a\}}(x)=f_{X\sqcup \{b\}}(x)+c\cdot (a-b)\cdot \prod \limits_{z\in X} (x-z).$$
where $c$ is the coefficient for $x^{|X|+1}$ in $f_{X\sqcup \{a,b\}}(x)$.
\end{Lemma}

\begin{proof} Clearly, $f_{X\sqcup \{a,b\}}(x)=f_{X\sqcup \{b\}}(x)+c\cdot (x-b) \prod \limits_{z\in X} (x-z)$, where $c$ is as in the statement of the lemma. Similarly,   $f_{X\sqcup \{a,b\}}(x)=f_{X\sqcup \{a\}}(x)+c\cdot (x-a) \prod \limits_{z\in X} (x-z)$. Subtracting one of these identities from the other proves the required formula.
\end{proof}

\begin{Lemma} Suppose $X$ is a circuit for $f$. Then

1)  $f_X(x)=cx^{|X|-1}+\ldots$ with $c\notin \mathbb Z$.

2) Suppose $d$ is the denominator of $c$ (in reduced form). Then all elements of $X$ are congruent modulo $d$.
\end{Lemma}

\begin{proof}  1) Take any $a\in X$. Then $f_X(x)=f_{X\setminus \{a\}}(x)+c\ \cdot \!\!\!\prod \limits_{z\in X\setminus \{a\}} (x-z)$ for some constant $c$, which is the coefficient of $x^{|X|-1}$. Since $f_{X\setminus \{a\}}\in \mathbb Z[x]$ but $f_X \notin \mathbb Z[x],$ $c \notin \mathbb Z$. 

2) For every $a$ and $b$ in $X$ by the Exchange Formula (Lemma 6) for $X\setminus\{a,b\}$, $a$, and $b$,  we have 
$$f_{X\setminus \{b\}}(x)=f_{X\setminus \{a\}}(x)+c\cdot (a-b)\cdot \prod \limits_{z\in X\setminus\{a,b\}} (x-z)$$

Since $f_{X\setminus \{b\}}$ and $f_{X\setminus \{a\}}$ are in $\mathbb Z[x],$ we have $c(a-b)\in \mathbb Z$, so $d\ |\ (a-b).$
\end{proof}

\begin{Remark} All functions described in Example 3 have only one circuit: $X=\{-1,1\}$. At the other extreme, a non-LIP function can have infinitely many circuits. We do not know if there is a function on $\mathbb Z$ with exactly $k\geq 2$ circuits.
\end{Remark}

\begin{Remark} All circuits of size $2$ are of the form $X=\{a,b\},$ where $a-b$ does not divide $f(a)-f(b)$. Larger circuits also exist, for example $f(x)=\frac{1}{2}x^2$ on $U=2\mathbb Z$ has $X=\{-2,0,2\}$ as a circuit. 
\end{Remark}

%%%%%%%%%%%%%%%%%%%%%%%%%%%%%
\section{Sheaf-theoretic properties of LIP Functions}

One of the fundamental properties of complex-analytic functions is that they form a sheaf with respect to the usual topology on $\mathbb C.$ We want to investigate the analogs of this phenomenon for LIP functions.

The most common topology on $\mathbb Z$ is the discrete topology: all sets are open and closed. But there are many natural topologies on $\mathbb Z$ (and its subsets) that have no finite non-empty open sets. This will be our assumption for the remainder of the paper.

\begin{Definition}
Suppose $S$ is an infinite subset of $\mathbb Z$, and $\tau$ is a topology on $S$, which we view as the collection of all open subsets of $S.$ Suppose that all non-empty $U\in \tau$ are infinite. Then the collection of rings $LIP(U)$ for all nonempty $U\in \tau$  together with the restriction maps is called the LIP presheaf on $(S,\tau)$. 
\end{Definition}

\begin{Remark} The formal definition of a presheaf, appropriate for our situation, is a contravariant functor from the category of nonempty $\tau$-open subsets of $S$ with respect to inclusion to the category of commutative rings. We will forgo such formalities.
\end{Remark}

\begin{Remark} The notions of presheaf and sheaf have played a pivotal role in the reformulation of Algebraic Geometry via the language of schemes by Grothendieck and his collaborators, and have been generalized considerably. For a relatively elementary introduction see, for example, \cite{GelfandManin}, section I.5.
\end{Remark}

Clearly, LIP presheaf on $(S,\tau)$ is a subpresheaf of the sheaf of all integer-valued functions on $\tau$-open nonempty subsets of $S.$ To every presheaf one can canonically associate a sheaf, by the construction called sheafification. In our situation it can be given by the following definition.

\begin{Definition}
For $(S,\tau)$ as above, the LIP sheaf of $(S,\tau)$, to be denoted by $\mathcal {LIP}(S,\tau)$ is defined as follows. For every nonempty $U\in \tau$ the ring $\mathcal {LIP}(S,\tau) (U)$ consists of all $f\colon U\to \mathbb Z$ such that for every $a\in U$ there exists an open  $U_a\subseteq U$ containing $a$, such  $f_{|_{U_a}}\in LIP(U_a).$ The restriction maps are the usual restriction maps. We will call elements of $\mathcal {LIP}(S,\tau) (U)$ locally LIP functions on $U$.
\end{Definition}

From the above definition, it is clear that $\mathcal {LIP}(S,\tau) (U)\supseteq LIP(U)$. In general, this inclusion is not an equality. In fact, if $U$ is not connected, then $\mathcal {LIP}(S,\tau) (U)$ is the product of  $\mathcal {LIP}(S,\tau) (U_i)$ where $U_i$ are the connected components of $U$. But if $U$ is connected, one can hope that the above inclusion is an equality. The main goal of this section is to prove this for several natural topologies. Our principal result is the following general theorem.

\begin{Theorem} Suppose $f\colon U\to \mathbb Z$, and $\mathbb Z\supseteq U=\bigcup_{\alpha} U_{\alpha}$, such that for every $\alpha_1$ and $\alpha_2$ and any distinct elements $a_1\in U_{\alpha_1}$ and $a_2\in U_{\alpha_2}$ the intersection $ U_{\alpha_1}\cap  U_{\alpha_2}$ contains infinitely many terms of the arithmetic progression $a_1+(a_2-a_1)\mathbb Z$. Then $f\in LIP(U)$ if and only if its restriction to each $U_{\alpha}$ is LIP.
\end{Theorem}

\begin{proof} The implication in one direction is obvious. So we suppose that all restrictions of $f$ to $U_{\alpha}$ are LIP, and we will prove that $f\in LIP(U)$ by means of contradiction.

Suppose $f$ is not LIP on $U$. Then it has some circuit $X$. For each $x\in X$ pick $U_{\alpha}$ that contains it. Replacing $U$ by the finite union of these $U_{\alpha}$, we can essentially reduce the problem to the case when the family $\{U_{\alpha}\}$ is finite: $\{U_1, U_2, \dots, U_N\}$. Suppose $k\leq N-1$ is the largest number such that the restriction of $f$ to $\bigcup_{i=1}^{k}U_i$ is LIP. Then the restriction of $f$ to $\bigcup_{i=1}^{k+1}U_i= (\bigcup_{i=1}^{k}U_i) \cup U_{k+1}$ is not LIP. So we have essentially reduced the problem to the case of a 2-set family: $U=V_1\cup V_2$, where $V_1=\bigcup_{i=1}^{k}U_i$ and $V_2=U_{k+1}$ (note that the intersection property is preserved).

Consider a circuit for $f$ on this $U=V_1\cup V_2$, that has the smallest possible number of elements not in $W=V_1\cap V_2$. Note that $X$ must have some element $a$ in $V_1\setminus V_2$ and some element $b$ in $V_2\setminus V_1$.   The minimality condition implies that if $a$ or $b$ is exchanged for any element  of $W\setminus X,$ the resulting set is not a circuit for $f$. Suppose $X=\{x_1,\dots , x_n,a,b\}$ and $w\in W\setminus X$. Then for some integers $s$ and $t$ we have

$$\left\{ \begin{array}{l}
f(w)= f_{X\setminus \{b\}}(w)+s\cdot (w-a) \prod \limits_{i=1}^n(w-x_i)\\
f(w)= f_{X\setminus \{a\}}(w)+t\cdot (w-b) \prod \limits_{i=1}^n(w-x_i)
\end{array}\right.$$

Subtracting and applying the Exchange Formula (Lemma 6), we get
$$0=c(a-b)\prod \limits_{i=1}^n(w-x_i)+\prod \limits_{i=1}^n(w-x_i) (s(w-a)-t(w-b)),$$
where $c\notin \mathbb Z$ is the highest degree coefficient of $f_X(x)$. Dividing by  $\prod \limits_{i=1}^n(w-x_i)$, we get $ 0=c(a-b)+s(w-a)-t(w-b).$ Because the intersection of $W$ and the arithmetic progression $a+(b-a)\mathbb Z$ is infinite, we can take $w$ in this progression to get a contradiction with $c\notin \mathbb Z$.
\end{proof}

The above theorem has applications to many topologies, including the following.

\begin{Definition} For $S\subseteq \mathbb Z$ the cofinite topology on $S$ is the topology where the non-empty open sets are the complements of finite sets.  
\end{Definition}

\begin{Definition} The density one topology on $\mathbb N$ is the topology where  the non-empty open sets are the sets of density one: $\lim \limits_{x\to +\infty} \frac{|\{n\in U  |  n\leq x\}|}{|\{n\in \mathbb N |  n\leq x\}|}=1.$
\end{Definition}

\begin{Definition} The density one topology on $\mathbb Z$ is the topology where  the non-empty open sets are the sets of density one: $\lim \limits_{x\to +\infty} \frac{|\{n\in U  |  |n|\leq x\}|}{|\{n\in \mathbb Z |  |n|\leq x\}|}=1.$
\end{Definition}

\begin{Corollary} For the cofinite and density one topologies on $\mathbb N$ and $\mathbb Z$ for all non-empty open sets $U$ we have the equality $\mathcal {LIP}(S,\tau) (U)= LIP(U)$. This means that the rings of  LIP functions form a sheaf in these topologies.
\end{Corollary}

\begin{proof} Intersection of any two non-empty open sets in any of the above topologies has density one, so it must intersect infinitely with every infinite arithmetic progression.
\end{proof}

\begin{Corollary} Suppose $P$ is the set of all primes in $\mathbb N$, with the  cofinite topology. Then for all non-empty open sets $U$ we have the equality $\mathcal {LIP}(S,\tau) (U)= LIP(U)$. This means that the rings of LIP functions form a sheaf in this topology.
\end{Corollary}

\begin{proof} The proof is a combination of Theorem 9 and Dirichlet's Theorem on primes in arithmetic progressions. Specifically, every arithmetic progression that contains two primes must contain infinitely many primes.
\end{proof}

\begin{Remark} The above example can be identified with the standard Zariski topology on the set of the maximal ideals of $\mathbb Z$. Instead of the cofinite topology we can also consider (natural or Dirichlet) density one topology, with essentially the same proof, but using a stronger version of Dirichlet's Theorem.
\end{Remark}

All of the above topologies are not Hausdorff, moreover all non-empty open sets are dense. A much more interesting example is given by Kirch topology on $\mathbb N$. We start with recalling its definition and some of its basic properties.

\begin{Definition} Kirch topology on $\mathbb N$ is the topology with the basis $\{a+d\mathbb Z_{\geq 0} \}$, where $a$ and $d$ are natural numbers, $gcd(a,d)=1,$ and $d$ is square-free.
\end{Definition}

Kirch topology makes $\mathbb N$ into a Hausdorff, connected, locally connected space. In fact, every open set from the above basis is connected (cf. \cite{Kirch}, Theorem 5, also \cite{Szczuka}, Theorem 3.5). From Dirichlet's Theorem on primes in arithmetic progression we get that the set of all primes is dense in $\mathbb N$ in Kirch topology.

For the reader's convenience, we state and prove the following easy lemma.

\begin{Lemma} Suppose $a_1,$ $a_2,$ $d_1,$ and $d_2$ are natural numbers. Then the arithmetic progressions $a_1+d_1\mathbb Z_{\geq 0}$ and  $a_2+d_2\mathbb Z_{\geq 0}$ intersect if and only if $gcd(d_1,d_2)$ divides $a_1-a_2.$ And when they do intersect, the intersection is of the form $a_3+d_3\mathbb Z_{\geq 0},$ where $d_3=lcm(d_1,d_2).$
\end{Lemma}

\begin{proof} Suppose first that there exists $a\in a_1+d_1\mathbb Z_{\geq 0} \cap a_2+d_2\mathbb Z_{\geq 0}$. Then $a=a_1+n_1d_1=a_2+n_2d_2$. Therefore $a_1-a_2=n_2d_2-n_1d_1,$ so $gcd(d_1,d_2)$ divides $a_1-a_2.$ In the other direction, suppose $gcd(d_1,d_2)$ divides $a_1-a_2$. Then by Bezout Lemma, there exist integers $n_1$ and $n_2$ such that $a_1-a_2=n_2d_2-n_1d_1$. This implies $a_1+n_1d_1=a_2+n_2d_2,$ By replacing the pair $(n_1, n_2)$ by $(n_1+kd_2,n_2+kd_1)$ for large enough $k$ we get that the number above is positive, so the two arithmetic progressions intersect.

To prove that the intersection is an arithmetic progression, suppose $a_3$ is the smallest element of the intersection. Then
$$a_1+d_1\mathbb Z_{\geq 0} \cap a_2+d_2\mathbb Z_{\geq 0}\supseteq  a_3+d_1\mathbb Z_{\geq 0} \cap a_3+d_2\mathbb Z_{\geq 0}=a_3+lcm(d_1,d_2) \mathbb Z_{\geq 0}.$$
On the other hand, if $a\in a_1+d_1\mathbb Z_{\geq 0} \cap a_2+d_2\mathbb Z_{\geq 0},$ then $a-a_3$ is divisible by both $d_1$ and $d_2$, thus by $lcm(d_1,d_2)$.
\end{proof}

The next lemma is undoubtedly classical, but we  have not seen it anywhere, and it seems both fundamental and somewhat counter-intuitive.

\begin{Lemma} Suppose $a\in a_1+d_1\mathbb Z_{\geq 0}$ and $b\in  a_2+d_2\mathbb Z_{\geq 0}$, where these arithmetic progressions intersect, and $a\neq b$. Then the intersection of these progressions intersects infinitely with the progression generated by $a$ and $b$ ($a+(b-a)\mathbb Z_{\geq 0}$ or $b+(a-b)\mathbb Z_{\geq 0}$).
\end{Lemma}

\begin{proof} Suppose $a=a_1+m_1d_1$ and $b=a_2+m_2d_2$. We can assume that $a<b.$ From the previous lemma, $gcd(d_1,d_2)$ divides $a_1-a_2$. So it also divides $b-a$. If the intersection of the progressions is $a_3+d_3\mathbb Z_{\geq 0},$ we get that $a_3=a+n_1d_1=b+n_2d_2$ for some integer $n_1$ and $n_2.$ So  $a_3-a$ is a multiple of $d_1.$ It also equals $a_3-b+(b-a)=n_2d_2+(b-a)$, thus it is a multiple of $gcd(b-a, d_2).$ So it is a multiple of $gcd(b-a, lcm(d_1,d_2)).$  The result now follows from the previous lemma.
\end{proof}

\begin{Corollary} Suppose $U_1$ and $U_2$ are intersecting infinite arithmetic progressions in $\mathbb N.$ Suppose $f\colon U_1\cup U_2 \to \mathbb Z$ is a function, whose restrictions to $U_1$ and $U_2$ are LIP. Then $f\in LIP(U_1\cup U_2)$. 
\end{Corollary}

\begin{proof} This is a combination of the previous lemma and Theorem 9.
\end{proof}

One would like to extend the above corollary to prove that every locally LIP function on a connected Kirch-open subset of $\mathbb N$ is LIP. However, it is unclear if this is true.  As a weaker version, one can hope that this is true for all basic arithmetic progressions, or, at least, for $\mathbb N$ itself. Most generally, we would like to know for what connected Kirch-open sets $U$ every locally LIP function on $U$ is LIP on $U$. At the moment, this question is wide open: we have no set $U$ for which the answer, positive or negative, is known.

\vskip20pt\noindent {\bf Acknowledgments.} The author thanks Sayak Sengupta for providing initial motivation for this work and for  helpful comments on a preliminary draft of this paper. The author also thanks Levi Axelrod whose question led him to the Golomb topology, that gave significant further motivation to study LIP functions on infinite subsets. %The author is also indebted to the anonymous referee for numerous helpful comments.

\end{document}